\documentclass[12pt]{amsart}
\usepackage{amssymb}
\usepackage{amsfonts}
 \usepackage{amsfonts,amssymb}

\usepackage{mathrsfs}
\let\mathcal\mathscr

\usepackage[all,ps,cmtip]{xy}

\def\llra{\hbox to 10mm{\rightarrowfill}}

\def\lllra{\hbox to 15mm{\rightarrowfill}}

\def\PA{{\widehat A}}
\def\PB{{\widehat B}}

\def\phi{{\varphi}}

\def\cF{\mathcal{F}}

\def\cO{\mathcal{O}}

\def\cM{\mathcal{M}}

\def\cQ{\mathcal{Q}}

\def\cV{\mathcal{V}}

\let\tilde\widetilde

\DeclareMathOperator{\rank}{rank}

\DeclareMathOperator{\codim}{codim}
\DeclareMathOperator{\Pic}{Pic}

\newtheorem{lemm}{Lemma}[section]
\newtheorem{theo}[lemm]{Theorem}

\newtheorem{prop}[lemm]{Proposition}

\newtheorem*{conj*}{Conjecture}

\theoremstyle{definition}

\newtheorem{rema}[lemm]{Remark}

\newtheorem{exam}[lemm]{Example}

\newtheorem{qu}[lemm]{Question}

\theoremstyle{remark}
\newtheorem*{remark*}{Remark}
\newtheorem*{note*}{Note}

\begin{document}
\title[M-regular Decompositions]{M-regular Decompositions for pushforwards of pluricanonical bundles of pairs to abelian varieties}
\author{Zhi Jiang}
\address{Shanghai center for mathematical sciences, Xingjiangwan campus, Fudan University, Shanghai 200438, P. R. China}
\email{zhijiang@fudan.edu.cn}
 \thanks{The author is partially supported by NSFC grants No. 11871155 and No. 11731004.}
\maketitle
\begin{abstract}
We extend the so called Chen-Jiang decompositions for pushforwards of pluricanocanical bundles to abelian varieties to the setting of klt pairs. We also provide a geometric application of this decomposition.
\end{abstract}
\section{Introduction}
Throughout this paper, we work over the complex number field $\mathbb C$.
Recall the following result, which is very useful in the study of irregular varieties.
\begin{theo}\label{CJ-form}Let $X$ be a smooth projective variety and $f: X\rightarrow A$ be a morphism from $X$ to an abelian variety. Then
\begin{itemize}
\item[(1)] for any $i\geq 0$, $$R^if_*\omega_X\simeq \bigoplus_{p_B: A\rightarrow B}\bigoplus_k p_{B}^*\cF_{B, k}\otimes Q_{B,k},$$ where $p_B$ are surjective morphisms between abelian varieties with connected fibers, $\cF_{B, k}$ are M-regular sheaves on $B$, and $Q_{B,k}$ are torsion line bundles;
\item[(2)] there exists a quotient of abelian varieties $p_0: A\rightarrow B_0$ such that for any $i\geq 2$,
$$f_*\omega_X^i=\bigoplus_kp_0^*\cF_{k}\otimes Q_{k},$$ where $\cF_k$ are M-regular sheaves on $B_0$ and $Q_k$ are torsion line bundles.
\end{itemize}
\end{theo}
\begin{rema}
\begin{itemize}
\item[(1)]The first statement was first proved in \cite{CJ} when $f$ is generically finite onto its image and the general form was proved in \cite{PPS} and the decomposition was referred as Chen-Jiang decomposition.  We will also adopt this terminology in this note. The formula also holds when $X$ is a compact K\"ahler manifold and $f: X\rightarrow A$ is a morphism from $X$ to a complex torus (see \cite{PPS}).
\item[(2)] The second statement was proved in \cite{LPS}.
\end{itemize}
\end{rema}

Both statements are quite powerful and it is then natural to ask if they holds for (pluri-)canonical bundles of pairs. The main result of this note is to extend both statements to the setting of klt pairs.
\begin{theo}Let $(X, \Delta)$ be a klt pair and let $f: X\rightarrow A$ be a  morphism from $X$ to an abelian variety $A$.
\begin{itemize}
\item[(1)]Suppose that a Cartier divisor $D\sim_{\mathbb Q} K_X+\Delta$. Then $R^jf_*\mathcal O_X(D)$ admits Chen-Jiang decomposition for each $j\geq 0$.
\item[(2)]Suppose that a Cartier divisor $D\sim_{\mathbb Q}m(K_X+\Delta)$  with non-negative Kodaira dimension for some $m\geq 2$. Then there exists a quotient between abelian varieties with connected fibers $p: A\rightarrow B$, $IT^0$ sheaves $\cF_i$ on $B$, and $Q_i\in \Pic^0(A)$ torsion line bundles, for $1\leq i\leq N$; such that $$f_*\mathcal O_X(D)=\bigoplus_{1\leq i\leq N}p^*\cF_i\otimes Q_i.$$
    \end{itemize}
\end{theo}
\begin{rema}
After this note was finished, the author was informed by Fanjun Meng that he proved essentially the same result by some different method. The author likes to thank him  by  sharing the preprint \cite{M}.
\end{rema}

\section{Preliminaries}
\subsection{Notations}
Let $f: X\rightarrow A$ be a morphism from a projective variety to an abelian variety and let $\cQ$ be a coherent sheaf on $X$. For any $i\geq 0$ and $j>0$, we can define $$V^i_j(\cQ, f):=\{P\in \Pic^0(A) \mid H^i(X, \cQ\otimes f^*P)\geq j\}.$$ By semicontinuity, $V^i_j(\cQ, f)$ is a closed subset of $\Pic^0(A)$ for any $i\geq 0$ and $j>0$. When $\cF$ is a coherent  sheaf on $A$, we   denote by $$V^i(\cF):=\{P\in \Pic^0(A) \mid H^i(A, \cF\otimes P)> 0\}$$ the $i$-th cohomological support loci of $\cF$ for $i\geq 0$.

The cohomological support loci of a coherent sheaf $\cF$ on $A$ is closely related to the positivity of the sheaf $\cF$.
We say that $\cF$  is \textit{IT$^0$} if $V^i(\cF)=\emptyset$ for $i>0$. We say that $\cF$ is \textit{M-regular} (resp. \textit{GV}) if $\codim_{\Pic^0(A)}V^i(\cF)>i$ (resp. $\codim_{\Pic^0(A)}V^i(\cF)\geq i$) for $i>0$.

M-regular sheaves were introduced in \cite{PP} and GV sheaves were studied in \cite{H} and \cite{PP2}. Hacon  proved in \cite{H} that for any morphism $f: X\rightarrow A$ from a smooth projective variety to an abelian variety,
$R^jf_*\omega_X$ is GV for any $j\geq 0$. Hacon also proved that if $\cF\neq 0$ is a GV sheaf on an abelian variety $A$, then $V^0(\cF)\neq \emptyset$ (see \cite{H}). Pareschi and Popa showed that if $\cF\neq 0$ is M-regular, then $V^0(\cF)=\Pic^0(A)$ (see \cite{PP}).

 We will say that a  sheaf $\cF$ on $A$ admits \textit{Chen-Jiang decomposition} if there exists    surjective morphisms between abelian varieties $q_i: A\rightarrow B_i$ with connected fibers, M-regular sheaves $\cF_i$ on $B_i$, and torsion line bundles $Q_i\in \Pic^0(A)$, for  $1\leq i\leq N$, such that $$\cF\simeq \bigoplus_{1\leq i\leq N}q_i^*\cF_i\otimes Q_i.$$

 It is straightforward to see that each direct summand $q_i^*\cF_i\otimes Q_i$ is actually GV on $A$. Thus if a sheaf admits Chen-Jiang decomposition, it is GV. By \cite{D} and \cite{PP3} we know that   GV sheaves are nef and  M-regular sheaves are ample. Hence if a  coherent sheaf $\cF$ on $A$ admits Chen-Jiang decomposition,   it is  semi-ample.

 We allow $q_i$ to be the identity morphism $A\rightarrow A$ or the trivial morphism from $A$ to $\mathrm{Spec}( \mathbb C)$. Thus by definition, a M-regular sheaf on $A$ admits Chen-Jiang decomposition and a direct sum of torsion line bundles also admits Chen-Jiang decomposition.

\subsection{Some useful results}
Our goal is to extend Theorem \ref{CJ-form} to singular settings.

In \cite[Variant 5.5]{PS}, Popa and Schnell proved that the pushforward of pluricanonical bundles of log canonical pairs are GV.
\begin{theo}\label{PS}
Let $k\geq 1$ be an integer and let $(X, \Delta)$ be a log canonical pair and $f: X\rightarrow A$ be a morphism to an abelian variety. Suppose that $k(K_X+\Delta)$ is Cartier, then $f_*\mathcal O_X(k(K_X+\Delta))$ is GV.
\end{theo}

\begin{rema}\label{rema1}
Popa and Schnell proved a vanishing result \cite[Theorem 1.7]{PS} for the pushforwards of Cartier divisors which are $\mathbb R$-equivalent to $m(K_X+\Delta)$, which, combining Hacon's criterion (\cite[Theorem 1.2]{H}),  implies Theorem \ref{PS}. Their arguments proved that for any Cartier divisor $D\sim_{\mathbb Q}k(K_X+\Delta)$, $f_*\cO_X(D)$ is GV.
\end{rema}

Shibata \cite{Sh} studied the cohomological support loci of pluricanonical bundles of log canonical pairs.
\begin{theo} \label{shibata}Let $(X, \Delta)$ be a log canonical pair and let $f: X\rightarrow A$ be a morphism from $X$ to an abelian variety $A$.
\begin{itemize}
\item Assume that $D$ is a Cartier divisor such that $D\sim_{\mathbb Q} m(K_X+\Delta)$ for some integer $m\geq 1$, then $V^0_j(\cO_X(D), f) $ is a finite union of torsion translates of abelian subvarieties of $\Pic^0(A)$ for all $j\geq 0$;
\item  If $(X, \Delta)$ is moreover log smooth, then for any Cartier divisor $D\sim_{\mathbb Q}K_X+\Delta$, then $V^i_j(\cO_X(D), f)$ is a finite union of torsion translates of abelian subvarieties of $\Pic^0(A)$ for all $i,j\geq 0$.
\end{itemize}
\end{theo}

These two results strongly imply that the pushforward of pluricanonical bundles of pairs should also admit Chen-Jiang decompositions.

The following small variant of  \cite[Corollary 2.13]{W} is also very crucial.
\begin{theo}\label{metric}
Let $f: X\rightarrow Y$ be a surjective morphism between  projective varieties. Assume that $Y$ is smooth and there exists  a $\mathbb Q$-effective divisor $\Delta$ on $X$ such that $(X, \Delta)$ is klt. Assume that for some integer $m>1$, there exists a Cartier divisor $D$ on $X$ such that $D\sim_{\mathbb Q}m(K_{X/Y}+\Delta)$ and $c_1(f_*\mathcal O_X(D))=0\in H^2(Y, \mathbb R)$, then $f_*\mathcal O_X(D)$ is a vector bundle with an Hermitian flat metric.
\end{theo}
\begin{proof}
Let $\mu: X'\rightarrow X$ be a log resolution of $(X, \Delta)$. We write $$K_{X'}+\Delta'+E_1=\mu^*(K_X+\Delta)+E_2,$$ where $\Delta'$ is the strict transform of $\Delta$, $E_1$ and $E_2$ are $\mu$-exceptional effective $\mathbb Q$-divisors without common components, and $\Delta'+E_1+E_2$ has SNC support. Since $(X, \Delta)$ is klt, we have $\lfloor \Delta'+E_1\rfloor=0$. Let $E_2'=\lceil E_2\rceil -E_2$. Then $$K_{X'}+\Delta'+E_1+E_2'=\mu^*(K_X+\Delta)+\lceil E_2\rceil.$$
Note that $(X', \Delta'+E_1+E_2')$ is still klt and $\mu^*D+m\lceil E_2\rceil\sim_{\mathbb Q} m\mu^*(K_{X/Y}+\Delta)+m\lceil E_2\rceil=m(K_{X'/Y}+\Delta'+E_1+E_2'). $ Moreover, $(f\circ\mu)_*\cO_{X'}(\mu^*D+m\lceil E_2\rceil)=f_*\cO_X(D)$, whose first Chern class is $0$ in $H^2(Y, \mathbb R)$.
Hence we may simply assume that $(X, \Delta)$ is klt, log smooth and $D\sim_{\mathbb Q}m(K_X+\Delta)$.

 By \cite[Theorem 1.12]{W}, it suffices to show that $f_*\cO_X(D)$ is equipped with a semi-positively curved singular Hermitian metric $h$ satisfying the $L^2$-extension property. This can be proved with exactly the same argument of the proof of \cite[Theorem B]{W}.  Let's be more explicit.
 Set $L:= D - m K_{X/Y}$. As $D\sim_{\mathbb Q}  m(K_{X/Y}+\Delta)$, $L$ can be equipped with a singular metric $h_L$ such that $i\Theta_{h_L} (L) = m \Delta$ in the sense of current (Here, by abusing a bit the notation, the RHS is the current associated to the $\mathbb{Q}$-divisor $m \Delta$).
Then the singularity of $h_L$ is of type $\frac{1}{\prod_i |s_i|^{m a_i }}$. (Here $\Delta = \sum a_i E_i$. $s_i$ is the canonical section of the divisor $E_i$ and $a_i \in \mathbb{Q}$).
In particular, we have $\mathcal{I} (h_L ^{1/m}) = \mathcal{O}_X$ since $\Delta$ is klt.

Therefore we can apply \cite[Thm 3.36]{P}
to $(D = m K_{X/Y} +L, h_L)$. There exists thus a possibly singular metric $h_{NS}$ on $f_*\cO_X (D)$ which is semi-positively curved. The metric $h_{NS}$ satisfies automatically the $L^2$-extension property.
We can thus apply \cite[Thm 1.12]{W} to $f_* \cO_X(D)$. Then $f_* \cO_X(D)$ is hermitian flat.
 \end{proof}

 We have the vanishing theorem for klt pairs (see \cite[Corollary 10.15 and Corollary 10.16]{K}).
\begin{theo}\label{Kollar}
Let $(X, \Delta)$ be a klt pair and let $f: X\rightarrow Y$ be a surjective morphism from a smooth projective variety to a normal projective variety. Assume that $M$ is a Cartier divisor on $X$ such that $M\equiv K_X+\Delta+N$ with $N$ a  $\mathbb Q$-Cartier divisor. Then,
\begin{itemize}
\item[(1)] if $N=f^*N_Y$ for some big and nef $\mathbb Q$-Cartier divisor $N_Y$ on $Y$, $H^i(Y, R^jf_*\mathcal O_X(M))=0$ for all $i>0$ and $j\geq 0$;
 \item[(2)] if $f$ is birational and $N$ is $f$-nef, then $R^if_*\cO_X(M)=0$ for $i>0$.
\end{itemize}
\end{theo}
This result was partially generalized by Ambro (\cite[Theorem 3.2]{A}) and Fujino (\cite[Theorem 6.3]{F}).
%\begin{theo}\label{KV-Fujino}Let $(X, \Delta)$ be a klt pair and let $L$ be a $\mathbb Q$-Cartier Weil divisor on $X$. Assume that $L-(K_X+\Delta)$ is nef and big over $V$, where $\pi: X\rightarrow V$ is a proper morphism. Then $R^q\pi_*\mathcal O_X(L)=0$ for any $q>0$.\end{theo}
\begin{theo}\label{KV-Fujino}[Ambro-Fujino] Let $f: X\rightarrow Y$  be a morphism between projective varieties, with $X$ smooth and $Y$ of dimension $n$. Let $(X, \Delta)$ be a log canonical, log smooth $\mathbb R$-pair, and consider a line bundle $B$ on X such that $B\sim_{\mathbb R} K_X+\Delta+f^*H$, where $H$ is an ample $\mathbb R$-Cartier $\mathbb R$-divisor on $Y$. Then $$H^i(Y, R^jf_*\mathcal O_X(B))=0$$ for all $i\geq 1$ and $j\geq 0$.
\end{theo}
\section{The proof of the main result}
\begin{prop}\label{canonical}
Let $(X, \Delta)$ be a klt pair and let $f: X\rightarrow A$ be a morphism from $X$ to an abelian variety $A$. Suppose that a Cartier divisor $D\sim_{\mathbb Q} K_X+\Delta$. Then $R^jf_*\mathcal O_X(D)$ admits Chen-Jiang decomposition for each $j\geq 0$.
\end{prop}
\begin{proof}
Take a log resolution $\mu: X'\rightarrow X$ such that $K_{X'}+\Delta'+E_1=\mu^*(K_X+\Delta)+E_2$, where $\Delta'$ is the strict transforms of $\Delta$, $E_1$ and $E_2$ are effective $\mu$-exceptional divisors without common components, and $\Delta'+E_1+E_2$ has SNC components. Note that $\lfloor \Delta'+E_1\rfloor=0$ since $(X, \Delta)$ is klt. Let $E_2':=\lceil E_2 \rceil-E_2$. We then have
 $$K_{X'}+\Delta'+E_1+E_2'=\mu^*(K_X+\Delta)+\lceil E_2\rceil.$$ Note that $(X', \Delta'+E_1+E_2')$ is still klt.
 Let $D':=\mu^*D+\lceil E_2\rceil$. Then $D'\sim_{\mathbb Q}K_{X'}+\Delta'+E_1+E_2'$.
 Since $\mu$ is birational, by Theorem \ref{Kollar}, we have
 \begin{eqnarray*}R\mu_*\mathcal O_{X'}(D')&=&\mu_*\mathcal O_{X'}(D')\\
 &=&\mathcal O_X(D).
 \end{eqnarray*}

Hence $R^if_*\mathcal O_X(D)=R^i(f\circ \mu)_*\mathcal O_{X'}(D')$.
Hence we may  simply assume that $X$ is a smooth projective variety and $\Delta$ has SNC support and $\lfloor\Delta \rfloor=0$.
Let $C=D-K_X$. Take the minimal positive integer $N$ such that $N\Delta$ is an effective Cartier divisor and $NC\sim N\Delta$. Take the resolution of singularities of the corresponding
$N$-cyclic cover, we get a generically finite morphism $\pi: Y\rightarrow X$. By \cite[Proposition 9.8]{K}, we know that $$\pi_*\cO_Y=\bigoplus_{k=0}^{N-1}\mathcal O_X(\lfloor k\Delta\rfloor-kC),$$ and by Grothendieck duality,  $$\pi_*\omega_Y=\omega_X\otimes\big(\bigoplus_{k=0}^{N-1}\mathcal O_X(kC-\lfloor k\Delta\rfloor)\big).$$ In particular, for $k=1$, we see that $\mathcal O_X(K_X+C)=\mathcal O_X(D)$ is a direct summand of $\pi_*\omega_Y$.
 Thus $R^i(f\circ \pi)_*\omega_Y$ has a direct summand $R^if_*\mathcal O_X(D)$. Since $R^i(f\circ \pi)_*\omega_Y$ admits Chen-Jiang decomposition, we conclude that $R^if_*\mathcal O_X(D)$ also admits  Chen-Jiang decomposition by \cite[Porposition 4.6]{LPS}.

\end{proof}
\begin{lemm}\label{klt}Let $(X, \Delta)$ be a klt pair and let $f: X\rightarrow B$ be a  morphism from $X$ to an abelian variety $B$. Suppose that a Cartier divisor $D\sim_{\mathbb Q}m(K_X+\Delta)$ for some $m\geq 2$ and the Iitaka model of $D$ dominates $f$. Then $f_*\cO_X(D)$ is IT$^0$
\end{lemm}
\begin{proof}
We may assume that $f_*\cO_X(D)\neq 0$, otherwise nothing needs to be proved. Then by Theorem \ref{PS}, $f_*\cO_X(D)$ is GV and hence $V^0(f_*\cO_X(D))\neq \emptyset$.

By Theorem \ref{shibata}, $V^0_j(\cO_X(D), f)$ is a union of torsion translates of abelian subvarieties of $\Pic^0(B)$. After replacing $D$ by $D\otimes f^*Q$ for some torsion line bundle $Q$, we may assume that $$0<h^0(X, \cO_X(D))=\max\{h^0(X, \cO_X(D)\otimes f^*P)\mid P\in \Pic^0(B)\}.$$

We first reduce this lemma to the case where $(X, \Delta)$ is klt, log smooth. We take a log resolution $\mu: X'\rightarrow X$ of $(X, \Delta)$. Write $$K_{X'}+\Delta'+E_1=\mu^*(K_X+\Delta)+E_2,$$ where $\Delta'$ is the strict transform of $\Delta$ and $E_1$ and $E_2$ are effective $\mu$-exceptional $\mathbb Q$-divisors without common components. Let $E_2'=\lceil E_2\rceil-E_2$. Then $(X', \Delta'+E_1+E_2')$  is log smooth. Since $(X, \Delta)$ is klt, $(X', \Delta'+E_1+E_2')$ is also klt. Note that $\mu^*D+m\lceil E_2\rceil\sim_{\mathbb Q}m(K_{X'}+\Delta'+E_1+E_2')$ and $$f_*\cO_X(D)=f_*\mu_*(\mu^*D+\lceil E_2\rceil).$$ Thus we may simply assume that the klt pair $(X, \Delta)$ is  log smooth.

We then take $M$ sufficiently large and divisible and write $$|M(m-1)D|=|L_{M(m-1)}|+F_{M(m-1)},$$ where $F_{M(m-1)}$ is the fixed divisor of $|M(m-1)D|$. After birational modifications, we may assume that $|L_{M(m-1)}|$ is basepoint free and induces the Iitaka fibration of  $D$, and $F_{M(m-1)}+\Delta$ has SNC support.
Let $|D|=|L_D|+F_D$ where $F_D$ is the fixed divisor of $|D|$. Then $M(m-1)F_D-F_{M(m-1)}$ is effective.

Let $H$ be a general member of $|L_{M(m-1)}|$, we then have
\begin{eqnarray*}D&&\sim_{\mathbb Q} m(K_X+\Delta)=K_X+\Delta+(m-1)(K_X+\Delta)\\&&\sim_{\mathbb Q} K_X+\Delta+\frac{1}{Mm}F_{M(m-1)}+\frac{1}{Mm}H.
\end{eqnarray*}

A priori, there is no reason that $(X, \Delta+\frac{1}{Mm}F_{M(m-1)})$ is still klt. We need to subtract some effective divisor.
We write $\Delta=\sum_{1\leq i\leq k}\alpha_i V_i$ and $\frac{1}{Mm}F_{M(m-1)}=\sum_{1\leq i\leq k}\beta_iV_i$, where $V_i$ are prime divisors, $0\leq \alpha_i< 1$, and $\beta_i\geq 0$. We may assume that $\alpha_i+\beta_i>0$ for all $1\leq i\leq k.$ Let $V:=\sum_{1\leq i\leq k}\lfloor\alpha_i+\beta_i\rfloor V_i.$ Then $$\Delta+\frac{1}{Mm}F_{M(m-1)}-V=\sum_{1\leq i\leq k}\lbrace\alpha_i+\beta_i\rbrace V_i$$ is an effective divisor and $(X, \Delta+\frac{1}{Mm}F_{M(m-1)}-V)$ is klt, log smooth. Since $H$ is a general member of a basepoint free linear system, $(X, \Delta+(\frac{1}{Mm}F_{M(m-1)}-V)+\frac{1}{Mm}H)$ is also  klt, log smooth.
Moreover, $$D-V\sim_{\mathbb Q} K_X+( \Delta+\frac{1}{Mm}F_{M(m-1)}-V)+\frac{1}{Mm}H.$$

By the assumption that the Iitaka model of $(X, D)$ dominates $B$, there  exists an ample $\mathbb Q$-divisor $H_B$ on $B$ and an effective $\mathbb Q$-divisor $H'$ such that $H-f^*H_B\sim_{\mathbb Q} H'$. We then take $\epsilon>0$ sufficiently small such that $(X,  \Delta+(\frac{1}{Mm}F_{M(m-1)}-V)+\epsilon H'+(\frac{1}{Mm}-\epsilon)H)$ is klt.
Then $$D-V\sim_{\mathbb Q} K_X+\Delta+(\frac{1}{Mm}F_{M(m-1)}-V)+\epsilon H'+(\frac{1}{Mm}-\epsilon)H+\epsilon f^*H_B.$$

By Theorem \ref{Kollar}, $f_*\cO_X(D-V)$ is an IT$^0$ sheaf on $B$.

On the other hand, we compare $V$ and $F_D$. We  write $$F_D=\sum_{1\leq i\leq k}\gamma_iV_i+F_D',$$ where $F_D'$ is also an effective integral divisor with no common components with $\Delta+F_{M(m-1)}$.  Note that $F_D$ is an integral divisor and $F_D-\frac{1}{M(m-1)}F_{M(m-1)}$ is effective.  Thus $\gamma_i\geq \frac{m}{m-1}\beta_i$ and since $\gamma_i$ is an integer, we have $\gamma_i\geq \lceil \beta_i\rceil$. Moreover, because  $0\leq \alpha_i<1$, we have $  \lceil \beta_i\rceil \geq\lfloor\alpha_i+\beta_i\rfloor$. Thus $F_D-V$ is an effective divisor. We have $H^0(X, \cO_X(D-V))\simeq H^0(X, \cO_X(D))$.

We then consider on $B$ the short exact sequence of coherent sheaves $$0\rightarrow f_*\cO_X(D-V)\rightarrow f_*\cO_X(D)\rightarrow \tau\rightarrow 0,$$
where $\tau$ is the quotient sheaf. We have shown that $f_*\cO_X(D-V)$ is IT$^0$. Moreover, by Theorem \ref{PS}, $f_*\cO_X(D)$ is GV. Since $V^i(\tau)\subset V^i(f_*\cO_X(D))\cup V^{i+1}(f_*\cO_X(D-V))$, we conclude that $\tau$ is also a GV sheaf on $B$. Moreover, since $f_*\cO_X(D-V)$ is IT$^0$, we have \begin{eqnarray*}&&h^0(B, f_*\cO_X(D-V)\otimes P)\\
&=&\chi(B,  f_*\cO_X(D-V)\otimes P)=\chi(B, f_*\cO_X(D-V))\\&=& h^0(B, f_*\cO_X(D-V))=h^0(B, \cO_X(D-V))=h^0(X, \cO_X(D)).
\end{eqnarray*}
From the maximality of $h^0(X, \cO_X(D))$, we conclude that $$H^0(B, f_*\cO_X(D-V)\otimes P)\simeq H^0(B, f_*\cO_X(D)\otimes P)$$ and thus $V^0(\tau)\subset V^1(f_*\cO_X(D-V))=\emptyset$.

We know that a GV sheaf $\tau$ with $V^0(\tau)=\emptyset$ is $0$ (see for instance \cite[Corollary 3.2]{H}). Thus $f_*\cO_X(D)=f_*\cO_X(D-V)$ is IT$^0$.
\end{proof}

\begin{prop}\label{pluricanonical}
Let $(X, \Delta)$ be a klt pair and let $f: X\rightarrow A$ be a  morphism from $X$ to an abelian variety $A$. Suppose that a Cartier divisor $D\sim_{\mathbb Q}m(K_X+\Delta)$ for some $m\geq 2$. Then there exists a quotient between abelian varieties with connected fibers $p: A\rightarrow B$, $IT^0$ sheaves $\cF_i$ on $B$, and $Q_i\in \Pic^0(A)$ torsion line bundles, for $1\leq i\leq N$; such that $$f_*\mathcal O_X(D)=\bigoplus_{1\leq i\leq N}p^*\cF_i\otimes Q_i.$$
\end{prop}
\begin{proof}
Let $g: X\dashrightarrow Y$ be the Iitaka fibration of $(X, D)$. After birational modifications, we may assume that $g$ is a morphism and let $F$ be a general fiber of $g$. Then $(F, \Delta|_F)$ is also a klt pair and $\kappa(F, D|_F)=\kappa(F, (K_F+\Delta|_F))=0$. Hence $f(F)$ is a translate of a fixed abelian subvariety $K$ of $A$ (see for instance \cite[Proposition 4.5]{WY}). Let $p: A\rightarrow B=A/K$ be the quotient. We have a commutative diagram:
\begin{eqnarray*}
\xymatrix{
X\ar[r]^{f}\ar[d]^g & A\ar[d]^p\\
Y\ar[r]^{h} & B.}
\end{eqnarray*}
By Theorem \ref{shibata}, there exists $Q_i$, $1\leq i\leq N$ torsion line bundles on $A$ such that $V^0(f_*\mathcal O_X(D))=\bigcup_{1\leq i\leq N}(-Q_i+B_i)$, where each $B_i$ is an abelian subvariety of $\Pic^0(A)$ and $-Q_i+B_i$, $1\leq i\leq N$ are   irreducible components of $V^0(f_*\mathcal O_X(D))$.

We claim that each $B_i$ is contained in a translate of $\Pic^0(B)$.

It is basically the context of \cite[Lemma 2.2]{CH}. Let's recall the argument here. Let $p^*$ be the dual of $p$ and consider the quotient morphism  $$\Pic^0(B)\hookrightarrow \Pic^0(A)\xrightarrow{\pi} \Pic^0(A)/\Pic^0(B)=\Pic^0(K).$$ It suffices to show that $\pi(B_i)$ is the origin of $\Pic^0(K)$ for each $1\leq i\leq N$. Assume the contrary,  $\pi(B_i)$ is a positive dimensional abelian subvariety of $\Pic^0(K)$. Since $H^0(X, \cO_X(D)\otimes f^*(Q_i^{-1}\otimes P))\neq 0$ for any $P\in B_i$,   we  have $$H^0(F, \cO_F(D|_F)\otimes   f^*(Q_i^{-1}\otimes P)|_F)\neq 0.$$ Moreover, the image of  $f|_F: F\rightarrow A$ is a translate of $K$, thus $f^*: \Pic^0(A)\rightarrow  \Pic^0(F)$ factors as
 $$\Pic^0(A)\xrightarrow{\pi} \Pic^0(K)\xrightarrow{f|_F^*} \Pic^0(F),$$ and $\Pic^0(K)\xrightarrow{f|_F^*} \Pic^0(F)$ is an isogeny onto its image. Therefore, there exists a positive dimensional abelian subvariety $K'$ of $\Pic^0(F),$ such that for any $Q\in K'$, we have $H^0(F, \cO_F(D|_F)\otimes f^*Q_i^{-1}|_F\otimes Q)\neq 0$. Let $M$ be the order of $Q_i\in \Pic^0(A)$. Then by a standard argument, we conclude that $h^0(F, \cO_F(MD|_F))>1$, which contradicts the fact that $\kappa(F, D|_F)=0$.\\

For any torsion line bundle $P\in V^0(f_*\mathcal O_X(D))$, $p_*(f_*\mathcal O_X(D)\otimes P)$ is non-trivial. By  Lemma \ref{klt},
$\cF_P:=p_*(f_*\mathcal O_X(D)\otimes P)$ is IT$^0$ on $B$.

Thus $V^0(f_*\mathcal O_X(D))$ is a disjoint union of torsion translates of $\Pic^0(B)$ and we may write $V^0(f_*\mathcal O_X(D))=\bigcup_{1\leq i\leq N}(-Q_i+\Pic^0(B))$ for some torsion line bundles $Q_i\in \Pic^0(A)$. Here we assume that $-Q_i+\Pic^0(B)$, $1\leq i\leq N$, are connected components of $V^0(f_*\cO_X(D))$.

Let $\cF_i:=p_*(f_*\mathcal O_X(D)\otimes Q_i^{-1})$ be the corresponding IT$^0$ sheaf on $B$. We now consider the natural map
\begin{eqnarray}
\phi: \cF:=\bigoplus_{1\leq i\leq N}p^*\cF_i\otimes Q_i\rightarrow f_*\mathcal O_X(D).
\end{eqnarray}
From the construction, we know that for any $P\in\Pic^0(A)$, $$H^0(A, \cF\otimes P)\xrightarrow{\phi}H^0(A, f_*\mathcal O_X(D)\otimes P)$$ is an isomorphism.

Moreover, we claim that $\varphi$ is injective.

 First, we show that $$f_*\mathcal O_X(D)\mid_{A_y}\simeq \bigoplus_{1\leq i\leq N}Q_i^{\oplus a_i}\mid_{A_y},$$ for $y\in B$ a general point of $(p\circ f)(X)$ and $A_y$ the corresponding fiber of $p$ over $y$. Let $X_y$ be the fiber of $p\circ f$ over $y$ and let $f_y: X_y\rightarrow A_y$ be the corresponding morphism.
By base change we know that $f_*\mathcal O_X(D)|_{A_y}$ is isomorphic to $f_{y*}\mathcal O_{X_y}(D_y)$. Since $(X_y, \Delta|_{X_y})$ is klt and $D|_{X_y}\sim_{\mathbb Q}m(K_{X_y}+\Delta|_{X_y})$, $f_{y*}\mathcal O_{X_y}(D_y)$ is a GV-sheaf by Theorem \ref{PS} and Remark \ref{rema1}. Moreover, by Lemma \ref{klt}, we see that $V^0(f_{y*}\cO_{X_y}(D_y))=\{-Q_i|_{A_y}\mid 1\leq i\leq N\}$.
Hence $\dim V^0(f_{y*}\mathcal O_{X_y}(D_y))=0$, thus $f_{y*}\mathcal O_{X_y}(D_y)$ is a flat vector bundle with $c_1(f_{y*}\mathcal O_{X_y}(D_y))=0$ (see for instance \cite[Corollary 3.2]{H}).
By Theorem \ref{metric},  $f_{y*}\mathcal O_{X_y}(D_y)$ is a direct sum of torsion line bundles. We then conclude from the structure of $V^0(f_*\mathcal O_X(D))$ that there exists $a_i\in \mathbb N$ such that  $$f_*\mathcal O_X(D)\mid_{A_y}\simeq \bigoplus_{1\leq i\leq N}Q_i^{\oplus a_i}\mid_{A_y}.$$ Thus $\mathrm{rk} (\cF_i)=a_i$ and $\phi$ is an injective map.

Let $\cQ:=f_*\mathcal O_X(D)/\cF$. We have the short exact sequence $$0\rightarrow \cF\rightarrow f_*\mathcal O_X(D)\rightarrow \cQ\rightarrow 0.$$ Since both $\cF$ and $f_*\mathcal O_X(D)$ are GV sheaves, so is $\cQ$. Moreover, $V^0(\cQ)\subset V^1(\cF)\cup V^0(f_*\mathcal O_X(D))$. Since $\cF$ is GV, we know that $V^1(\cF)\subset V^0(\cF)$ (see for instance \cite[Corollary 3.2]{H}). Then   $V^0(\cQ)\subset V^0(\cF)=\bigcup_{1\leq i\leq N}(-Q_i+\Pic^0(B))$.
 For all $Q\in V^0(\cF)$, we know that $p_*(\cF\otimes Q)=p_*(f_*\mathcal O_X(D)\otimes Q)$. Hence we have the exact sequence $$0\rightarrow p_*(\cQ\otimes Q)\rightarrow R^1p_*(\cF\otimes Q)\rightarrow R^1p_*(f_*\mathcal O_X(D)\otimes Q).$$ We may assume that $Q=-Q_1+Q_2\in -Q_1+\Pic^0(B)$, where $Q_2\in \Pic^0(B)$, then $R^1p_*(\cF\otimes Q)=(\cF_1\otimes Q_2)^{\oplus h^1(\mathcal O_K)}$ is torsion free on its support. Since the map $$R^1p_*(\cF\otimes Q)\rightarrow R^1p_*(f_*\mathcal O_X(D)\otimes Q)$$ is isomorphism on the generic point of their support, we conclude that this map is injective. Hence $p_*(\cQ\otimes Q)=0$ for all $Q\in V^0(\cF)$. Thus $V^0(\cQ)=\emptyset$ and hence $\cQ=0$.
 \end{proof}

\begin{qu}
Does Chen-Jiang decompositions hold for pushforwards to abelian varieties of (pluri-)canonical bundles of log canonical pairs?
\end{qu}

We have the vanishing result Theorem \ref{KV-Fujino} for log canonical pairs. Thus Lemma \ref{klt} may still be true under some extra assumptions about the positivity of $D$. However, it seems difficult to extend Theorem \ref{metric} to the log canonical cases.

\section{A geometric application}

\begin{prop}\label{full}Let $(X, \Delta)$ be a klt pair. Assume that a big  Cartier  divisor $D\sim_{\mathbb Q}K_X+\Delta$ and there exists a morphism $f: X\rightarrow A$ which is generically finite onto its image. Then $V^0(f_*\mathcal O_X(D))$ generates $\Pic^0(A)$.
\end{prop}
\begin{proof}
This is a direct generalization of Chen and Hacon's theorem (see \cite{CH}) in the smooth case. The same argument works here. Assume that $V^0(f_*\mathcal O_X(D))$ generates an abelian subvariety $\PB$ of $\Pic^0(A)$. Considering the morphisms:
\begin{eqnarray*}
\xymatrix{
X\ar[dr]_g\ar[r]^f & A\ar[d]^p\\
& B,}
\end{eqnarray*}
where $p$ is the dual morphism of the inclusion $\PB\hookrightarrow \PA$ and $g=p\circ f$. Let $X_y$ be a connected component of a general fiber of $g$ and let $K$ be the kernel of $p$ and let $f_y: X_y\rightarrow K$ be the corresponding morphism. Then by Proposition \ref{canonical} and base change, $f_{y*}\mathcal O_{X_y}(D|_{X_y})$ is a direct sum of torsion line bundles. After \'etale cover of $K$ and \'etale base change, we may assume that $f_y$ is primitive, i.e. $f_y^*: \Pic^0(K)\rightarrow \Pic^0(X_y)$ is injective. Then $f_{y*}\mathcal O_{X_y}(D|_{X_y})=\mathcal O_{K}$ and hence $f_y: X_y\rightarrow K$ is birational. We also have $H^{\dim X_y}(X_y, \mathcal O_{X_y}(D|_{X_y}))=H^{\dim X_y}(K, \cO_K)\neq 0$.

Since $X_y$ is a connected component of a general fiber of $g$, $(X_y, \Delta|_{X_y})$ is still klt. Hence $X_y$ has rational singularities and in particular $X_y$ is Cohen-Macaulay (see for instance \cite[Theorem 5.22]{KM}).
Hence $0\neq H^{\dim X_y}(X_y, \mathcal O_{X_y}(D|_{X_y}))\simeq \mathrm{Hom}(\mathcal O_{X_y}(D|_{X_y}), \cO_{X_y}(K_{X_y}))^{\vee}$. Since $D|_{X_y}\sim_{\mathbb Q} K_{X_y}+\Delta|_{X_y}$, we conclude that $\Delta|_{X_y}=0$ and $K_{X_y}\sim D|_{X_y} $ is Cartier.
Note that $D$ is big on $X$ and the family of $X_y$ covers $X$, hence $D|_{X_y}$ is also big and so is $K_{X_y}$. This is a contradiction since we have shown that $f_y: X_y\rightarrow K$ is birational.
\end{proof}

\subsection*{Acknowledgements}
The author started to think about M-regular decomposition for klt pairs after a discussion with Mihnea Popa and he is grateful to Mihnea for asking this interesting question.  The author is  grateful to Junyan Cao for pointing out the reference \cite{W} and many stimulating conversations. The author   thanks Fanjun Meng for sharing his preprint \cite{M}, thanks Giuseppe Pareschi for some interesting comments. Finally, the author is grateful to an anonymous referee for his/her careful reading.


\begin{thebibliography}{24}
\bibitem[A]{A} Ambro, Florin. Quasi-log varieties. Tr. Mat. Inst. Steklova 240 (2003),220--239; reprinted in Proc. Steklov Inst. Math. 2003, no. 1(240), 214--233
\bibitem[CH]{CH} Chen, Jungkai Alfred;  Hacon, Christopher D.  Pluricanonical maps of varieties of maximal Albanese dimension, Math. Ann. 320, (2001), 367--380.
 \bibitem[CJ]{CJ}  Chen, Jungkai Alfred; Jiang, Zhi. Positivity in varieties of maximal Albanese dimension. J. Reine Angew. Math. 736 (2018), 225--253.
\bibitem[D]{D} Debarre, Olivier. On coverings of simple abelian varieties. Bull. Soc. Math. France 134 (2006), no. 2, 253--260.
\bibitem[F]{F}  Fujino, Osamu. Fundamental theorems for the log minimal model program. Publ. Res. Inst. Math. Sci. 47 (2011), no. 3, 727--789.
\bibitem[H]{H} Hacon, Christopher D. A derived category approach to generic vanishing. J. Reine Angew. Math. 575 (2004), 173--187.
 \bibitem[Ka]{Ka} Kawamata, Yujino. Characterization of abelian varieties. Compositio Math. 43 (1981), no. 2, 253--276.
 \bibitem[K1]{K} Koll\'ar, J\'anos, Shafarevich maps and Automorphic Forms. Princeton University Press.
 \bibitem[K2]{K2} Koll\'ar, J\'anos. Higher direct images of dualizing sheaves. I. Ann. of Math. (2) 123 (1986), no. 1, 11--42.
 \bibitem[KM]{KM} Koll\'ar, J\'anos, Mori, Shigefumi. Birational  geometry of algebraic varieties, Cambridge Tracts in Mathematics, vol. {\bf 134}, Cambridge University Press, Cambridge,1998, With the collaboration of C. H. Clemens and A. Corti, Translated from the 1998 Japanese original.
 \bibitem[L]{L} Lazarsfeld, Robert. Positivity in algebraic geometry. II. Classical setting: line bundles and linear series. Ergebnisse der Mathematik und ihrer Grenzgebiete. 3. Folge. A Series of Modern Surveys in Mathematics [Results in Mathematics and Related Areas. 3rd Series. A Series of Modern Surveys in Mathematics], 49. Springer-Verlag, Berlin, 2004.
 \bibitem[LPS]{LPS} Lombardi, Luigi; Popa, Mihnea; Schnell, Christian. Pushforwards of pluricanonical bundles under morphisms to abelian varieties.
To appear in Journal of the European Mathematical Society.
\bibitem[M]{M} Meng, Fanjun. Pushforwards of klt pairs under morphisms to abelian varieties. arXiv:2005.13761.
\bibitem[PP1]{PP}  Pareschi, Giuseppe; Popa, Mihnea. Regularity on abelian varieties. I. J. Amer. Math. Soc. 16 (2003), no. 2, 285--302.
\bibitem[PP2]{PP2}  Pareschi, Giuseppe; Popa, Mihnea. GV-sheaves, Fourier-Mukai transform, and generic vanishing. Amer. J. Math. 133 (2011), no. 1, 235--271.
\bibitem[PP3]{PP3} Pareschi, Giuseppe; Popa, Mihnea. Regularity on abelian varieties III: relationship with generic vanishing and applications. Grassmannians, moduli spaces and vector bundles, 141--167, Clay Math. Proc., 14, Amer. Math. Soc., Providence, RI, 2011.
\bibitem[PPS]{PPS}  Pareschi, Giuseppe; Popa, Mihnea; Schnell, Christian. Hodge modules on complex tori and generic vanishing for compact K\"ahler manifolds. Geom. Topol. 21 (2017), no. 4, 2419--2460.
  \bibitem[P]{P} P\u{a}un, Mihai.  Singular Hermitian metrics and positivity of direct images of pluricanonical bundles, https://arxiv.org/abs/1606.00174.

\bibitem[PS]{PS} Popa, Mihnea; Schnell, Christian. On direct images of pluricanonical bundles,
Algebra Number Theory 8 (2014), no. 9, 2273--2295.
\bibitem[Sh]{Sh}  Shibata, Takahiro. On generic vanishing for pluricanonical bundles. Michigan Math. J. 65 (2016), no. 4, 873--888.
\bibitem[WJ]{W} Wang, Juanyong. On the Iitaka Conjecture $C_{n,m}$ for K\"ahler Fibre Spaces, arXiv:1907.06705.
\bibitem[WY]{WY} Wang, Yuan. On the characterization of abelian varieties for log pairs in zero and positive characteristic, arXiv:1610.05630v2.

\end{thebibliography}
\end{document}